\documentclass[10pt,reqno]{amsart}

\usepackage{amsmath, amsthm, amscd, amssymb, amsfonts}

\newcommand{\Rb}{{\mathbb R}}
\newcommand{\Cb}{{\mathbb C}}

\newtheorem{thm}{Theorem}[section]
\newtheorem*{thm*}{Theorem}
\newtheorem{cor}[thm]{Corollary}
\newtheorem*{cor*}{Corollary}
\newtheorem{lem}[thm]{Lemma}

\newtheorem*{con*}{Conjecture}
\newtheorem*{prob*}{Problem}

\theoremstyle{definition}
\newtheorem{defn}[thm]{Definition}

\theoremstyle{remark}
\newtheorem{rem}[thm]{Remark}

\begin{document}
\title{Strict $p$-negative type of a metric space}
\author[Hanfeng Li]{Hanfeng Li$\,^\flat$}\thanks{$^\flat$ Partially supported by NSF grant DMS-0701414.}
\address{Department of Mathematics, SUNY at Buffalo, Buffalo, NY 14260, U.S.A.}
\email{hfli@math.buffalo.edu}
\author[Anthony Weston]{Anthony Weston$\,^\natural$}\thanks{$^\natural$ Partially supported by
internal research grants at Canisius College.}
\address{Department of Mathematics \& Statistics, Canisius College,
Buffalo, NY 14208, U.S.A.}
\email{westona@canisius.edu}


\begin{abstract}
Doust and Weston \cite{DW} have introduced a new method called \textit{enhanced negative type}
for calculating a non-trivial lower bound $\wp_{T}$ on the supremal strict
$p$-negative type of any given finite metric tree $(T,d)$. In the context of
finite metric trees any such lower bound $\wp_{T} > 1$ is deemed to be non-trivial.
In this paper we refine the technique of enhanced negative type and show how it
may be applied more generally to any finite metric space $(X,d)$ that is known
to have strict $p$-negative type for some $p \geq 0$. This allows us to significantly
improve the lower bounds on the supremal strict $p$-negative type of finite metric
trees that were given in \cite[Corollary 5.5]{DW} and, moreover, leads in to one
of our main results: The supremal $p$-negative type of a finite metric space
cannot be strict. By way of application we are then able to exhibit large classes
of finite metric spaces (such as finite isometric subspaces of Hadamard
manifolds) that must have strict $p$-negative type for some $p > 1$. We also show
that if a metric space (finite or otherwise) has $p$-negative type for some
$p > 0$, then it must have strict $q$-negative type for all $q \in [0,p)$. This
generalizes Schoenberg \cite[Theorem 2]{S2} and leads to a complete classification
of the intervals on which a metric space may have strict $p$-negative type.
\end{abstract}
\maketitle

\section{Introduction and Synopsis}
The study of positive definite kernels and the related notion of $p$-negative type metrics
dates back to the early 1900s with some antecedents in the 1800s. A major theme that
emerged was the search for metric characterizations of subsets of Hilbert space up to isometry.
Significant initial results on this classical embedding problem were obtained
by Cayley \cite{C}, Menger \cite{M1, M2, M3} and Schoenberg \cite{S1, S2, S3}.
We note in particular \cite[Theorem 1]{S3}:
A metric space is isometric to a subset of Hilbert space if and only if it has
$2$-negative type. This result was spectacularly generalized to the category of normed
spaces by Bretagnolle et al.\ \cite[Theorem 2]{BDK}:
A real normed space is linearly isometric to a subspace of some
$L_{p}$-space ($1 \leq p \leq 2$) if and only if it has $p$-negative type.
It remains a prominent question to give a complete generalization of this result to the
setting of non-commutative $L_{p}$-spaces. See, for example, Junge \cite{J}.

More recently, difficult questions concerning $p$-negative type metrics have figured
prominently in theoretical computer science. A prime example is the recently refuted
\textit{Goemans-Linial conjecture}: Every metric space of $1$-negative type bi-Lipschitz
embeds into some $L_{1}$-space. Although this conjecture clearly holds for normed spaces
by \cite[Theorem 2]{BDK} (with $p=1$), it is not true for arbitrary metric spaces.
This was shown by Khot and Vishnoi \cite{KV}. Subsequently, Lee
and Naor \cite{LN} have shown that there is no metric version of \cite[Theorem 2]{BDK}
(modulo bi-Lipschitz embeddings) for any $p \in [1,2)$. More precisely,
we have \cite[Theorem 1.2]{LN}: For each $p \in [1,2)$ there is a metric space $(X,d)$
of $p$-negative type which does not bi-Lipschitz embed into any $L_{p}$-space.

The related notion of strict $p$-negative type has been studied rather less well than its
classical counterpart and most known results deal with the case $p=1$.
The present work is motivated by functional analytic questions
that arise naturally from the papers of Hjorth et al.\ \cite{HKM, HLM}. Both \cite{HKM}
and \cite{HLM} focus on examples and properties of finite metric spaces of strict $1$-negative type.
One theme of these papers is to determine global geometric properties of finite metric spaces
of strict $1$-negative type. As an example we mention \cite[Theorem 3.9]{HLM}:
If a finite metric space is of strict $1$-negative type, then it has a unique $\infty$-extender.
It is also natural to ask for conditions on a finite metric space which will guarantee that
it has strict $1$-negative type. One such result is \cite[Theorem 5.2]{HLM}:
If a finite metric space is hypermetric and regular, then it is of strict $1$-negative type.

The theme of this paper is to determine basic properties of strict $p$-negative
metrics for all $p \geq 0$. In particular, we aim to move beyond the familiar case $p=1$,
thereby setting up the rudiments of a basic theory of strict $p$-negative type metrics.

Section $2$ is dedicated to a review of the salient features of generalized roundness, negative
type, and strict negative type.
In Definition \ref{NGAP} we recall the notion of the (normalized)
$p$-negative type gap $\Gamma_{X}^{p}$ of a metric space $(X,d)$. This parameter was recently
introduced by Doust and Weston \cite{DW, DW2} in order to obtain non-trivial  lower bounds on
the maximal $p$-negative of finite metric trees. Basic properties of $\Gamma_{X}^{p}$ will play
a vital r\^{o}le in our computations in Section $3$.

The observation is made in \cite[Theorem 5.2]{DW} that if the $p$-negative type gap
$\Gamma_{X}^{p}$ of a finite metric space $(X,d)$ is positive for some $p \geq 0$, then $(X,d)$
must have strict $q$-negative type on some interval of the form $[p, p + \zeta)$ where $\zeta > 0$.
However, the authors only provide an explicit value for $\zeta$ in the case $p = 1$. Letting $n=|X|$,
the value of $\zeta$ given in this case is $O(1/n^{2})$. (See \cite[Theorem 5.1]{DW}.)
The purpose of Section $3$ is to give a precise quantitative version of \cite[Theorem 5.2]{DW}
which yields significantly improved values of $\zeta$ for all $p \geq 0$.
In fact, for each $p \geq 0$, our value of $\zeta$ is $O(1)$.
The precise statement of this result is given in Theorem \ref{mainthm}. By way of application,
Theorem \ref{mainthm} leads to significantly improved lower bounds on the maximal $p$-negative type
of finite metric trees. These are given in Corollary \ref{treebound}. Then in Remark \ref{treerem}
we point out that the estimates given in Corollary \ref{treebound} are reasonably sharp
for finite metric trees that resemble stars.
This suggests there is little room for improvement in the statement of
Theorem \ref{mainthm} (in general).

In Section $4$ we use Theorem \ref{mainthm} and an elementary compactness argument to derive a key
result of this paper: The supremal $p$-negative type of a finite metric space cannot be strict. This is
done in Corollary \ref{strict} to Theorem \ref{mainthm2}.
Using known results we are then able to exhibit large classes of finite
metric spaces, all of which must have strict $p$-negative type for some $p > 1$. For example, any finite
isometric subspace of a Hadamard manifold must have strict $p$-negative type for some $p > 1$.
We collate an array of such examples in Corollary \ref{strictexs}.

The main results of Section $5$ are Theorems \ref{sint}, \ref{finint} and \ref{infint}.
These theorems generalize \cite[Theorem 2]{S2}. For example, in Theorem \ref{sint}
we show that if a metric space (finite or otherwise) has $p$-negative type for
some $p > 0$, then it must have strict $q$-negative type for all $q \in [0,p)$.
This allows us to precisely codify the types of intervals on which a metric space
may have strict $p$-negative type. It is interesting to note that finite metric spaces
behave quite differently to infinite metric spaces in this respect.
These differences are highlighted in Theorems \ref{finint} and \ref{infint}.
Understanding how strict negative type behaves on intervals
leads to further examples of metric spaces that have non-trivial strict $p$-negative type.
We then conclude the paper with the observation in Remark \ref{lastrem} that Theorems \ref{mainthm},
\ref{mainthm2} and \ref{sint} (as well as several of our corollaries) actually hold more generally for
finite semi-metric spaces. This is because we do not use the triangle inequality at any point in our
definitions or proofs.

Throughout this paper the set of natural numbers $\mathbb{N}$ is taken to consist of all positive integers
and sums indexed over the empty set are always taken to be zero. Given a real number $x$, we are using
$\lfloor x \rfloor$ to denote the largest integer that does not exceed $x$, and $\lceil x \rceil$ to
denote the smallest integer which is not less than $x$.

\section{A framework for ordinary and strict $p$-negative type}
We begin by recalling some theoretical features of (strict) $p$-negative type and its relationship
to (strict) generalized roundness. More detailed accounts may be found in
the work of Benyamini and Lindenstrauss
\cite{BL}, Deza and Laurent \cite{DL}, Prassidis and Weston \cite{PW}, and Wells and Williams \cite{WW}.
These works emphasize the interplay between the classical $p$-negative type inequalities and isometric,
Lipschitz or uniform embeddings. They also indicate applications to more contemporary areas of interest
such as theoretical computer science.
One of the most important results for our purposes is the equivalence
of (strict) $p$-negative type and (strict) generalized roundness $p$.
These equivalences are described in Theorem \ref{REMGR}.

\begin{defn}\label{types} Let $p \geq 0$ and let $(X,d)$ be a metric space. Then:
\begin{enumerate}
\item[(a)] $(X,d)$ has $p$-{\textit{negative type}} if and only if for all natural numbers $k \geq 2$,
all finite subsets $\{x_{1}, \ldots , x_{k} \} \subseteq X$, and all choices of real numbers $\eta_{1},
\ldots, \eta_{k}$ with $\eta_{1} + \cdots + \eta_{k} = 0$, we have:

\begin{eqnarray}\label{ONE}
\sum\limits_{1 \leq i,j \leq k} d(x_{i},x_{j})^{p} \eta_{i} \eta_{j}  & \leq & 0.
\end{eqnarray}

\item[(b)] $(X,d)$ has \textit{strict} $p$-{\textit{negative type}} if and only if it has $p$-negative type
and the associated inequalities (\ref{ONE}) are all strict except in the trivial case
$(\eta_{1}, \ldots, \eta_{k})$ $= (0, \ldots, 0)$.
\end{enumerate}
\end{defn}

\begin{rem} Every metric space obviously has strict $0$-negative type. It is also the case that
every finite metric space has strict $p$-negative type for some $p > 0$.
This follows from Weston \cite[Theorem 4.3]{W}, Theorem \ref{REMGR} (a) and Theorem \ref{sint}.
\end{rem}
\noindent It is possible to reformulate both ordinary and strict $p$-negative type in terms of an
invariant known as \textit{generalized roundness} from the uniform theory of Banach spaces. Generalized roundness
was introduced by Enflo \cite{E} in order to solve (in the negative) \textit{Smirnov's Problem}: Is
every separable metric space uniformly homeomorphic to a subset of Hilbert space? The analog of this
problem for coarse embeddings was later raised by Gromov \cite{G} and solved negatively by
Dranishnikov et al.\ \cite{DGLY}. Prior to introducing generalized roundness in Definition
\ref{GRUNT} (a) we shall develop some intermediate technical notions in order to streamline the
exposition in the remainder of this paper.

\begin{defn}\label{CUTIE}
Let $s,t$ be arbitrary natural numbers and let $X$ be any set.
\begin{enumerate}
\item[(a)] An $(s,t)$-\textit{simplex} in $X$ is an $(s+t)$-vector $(a_{1}, \ldots , a_{s}, b_{1}, \ldots ,b_{t})
\in X^{s+t}$ consisting of $s+t$ pairwise distinct coordinates
$a_{1}, \ldots, a_{s}, b_{1}, \ldots , b_{t} \in X$. Such a simplex will be denoted by $D=[a_{j};b_{i}]_{s,t}$.

\item[(b)] A \textit{load vector} for an $(s,t)$-simplex $D=[a_{j};b_{i}]_{s,t}$ in $X$ is an arbitrary vector
$\vec{\omega} = (m_{1}, \ldots m_{s}, n_{1}, \ldots , n_{t}) \in \mathbb{R}^{s+t}_{+}$ that assigns a positive weight
$m_{j} > 0$ or $n_{i} > 0$ to each vertex $a_{j}$ or $b_{i}$ of $D$, respectively.

\item[(c)] A \textit{loaded} $(s,t)$-\textit{simplex} in $X$ consists of an $(s,t)$-simplex $D=[a_{j};b_{i}]_{s,t}$ in $X$
together with a load vector $\vec{\omega} =(m_{1}, \ldots ,m_{s}, n_{1}, \ldots, n_{t})$ for $D$. Such a loaded simplex
will be denoted by $D(\vec{\omega})$ or $[a_{j}(m_{j});b_{i}(n_{i})]_{s,t}$ as the need arises.

\item[(d)] A \textit{normalized} $(s,t)$-\textit{simplex} in $X$ is a loaded $(s,t)$-simplex $D(\vec{\omega})$ in $X$ whose
load vector $\vec{\omega}=(m_{1}, \ldots , m_{s}, n_{1}, \ldots , n_{t})$ satisfies the two normalizations:
\[
m_{1} + \cdots + m_{s} = 1 = n_{1} + \cdots n_{t}.
\]
Such a vector $\vec{\omega}$ will be called a \textit{normalized load vector} for $D$.
\end{enumerate}
\end{defn}
\noindent Rather than giving the original definition of generalized roundness $p$ from \cite{E}, we shall
present an equivalent reformulation in Definition \ref{GRUNT} (a) that is due to Lennard et al.\
\cite{LTW} and Weston \cite{W}.

\begin{defn}\label{GRUNT}
Let $p \geq 0$ and let $(X,d)$ be a metric space. Then:
\begin{enumerate}
\item[(a)] $(X,d)$ has \textit{generalized roundness} $p$ if and only if
for all $s,t \in \mathbb{N}$ and all normalized $(s,t)$-simplices $D(\vec{\omega})
= [a_{j}(m_{j});b_{i}(n_{i})]_{s,t}$
in $X$ we have:
\begin{eqnarray}\label{TWO}
& & \sum\limits_{1 \leq j_{1} < j_{2} \leq s} m_{j_{1}}m_{j_{2}}d(a_{j_{1}},a_{j_{2}})^{p} +
\sum\limits_{1 \leq i_{1} < i_{2} \leq t} n_{i_{1}}n_{i_{2}}d(b_{i_{1}},b_{i_{2}})^{p} \nonumber \\
& \leq & \sum\limits_{j,i=1}^{s,t} m_{j}n_{i}d(a_{j},b_{i})^{p}.
\end{eqnarray}

\item[(b)] $(X,d)$ has \textit{strict generalized roundness} $p$ if and only if it has generalized
roundness $p$ and the associated inequalities (\ref{TWO}) are all strict.
\end{enumerate}
\end{defn}
\noindent Two key aspects of generalized roundness for the purposes of this paper are the following equivalences.

\begin{thm}[\cite{LTW}, \cite{DW}]\label{REMGR} Let $p \geq 0$ and let $(X,d)$ be a metric space. Then:
\begin{enumerate}
\item[(a)] $(X,d)$ has $p$-negative type if and only if
it has generalized roundness $p$.

\item[(b)] $(X,d)$ has strict $p$-negative type if and only if
it has strict generalized roundness $p$.
\end{enumerate}
\end{thm}

\noindent Based on Definition \ref{GRUNT} (a) and Theorem \ref{REMGR} we introduce two numerical
parameters $\gamma_{D}^{p}(\vec{\omega})$
and $\Gamma_{X}^{p}$ that are designed to quantify the \textit{degree of strictness}
of the non-trivial $p$-negative type inequalities.

\begin{defn}\label{SGAP} Let $p \geq 0$ and let $(X,d)$ be a metric space. Let $s,t$ be natural numbers and
$D=[a_{j};b_{i}]_{s,t}$ be an $(s,t)$-simplex in $X$. Denote by $N_{s,t}$
the set of all normalized load vectors $\vec{\omega}=
(m_{1}, \ldots , m_{s}, n_{1}, \ldots , n_{t}) \subset \mathbb{R}^{s+t}_{+}$ for $D$. Then
the \textit{(normalized)} $p$-\textit{negative type simplex gap} of $D$ is defined to be the
function $\gamma_{D}^{p} : N_{s,t} \rightarrow \mathbb{R}$ where
\begin{eqnarray*}
\gamma_{D}^{p}(\vec{\omega}) & = & \sum\limits_{j,i = 1}^{s,t} m_{j}n_{i}d(a_{j},b_{i})^{p}
- \sum\limits_{1 \leq j_{1} < j_{2} \leq s} m_{j_{1}}m_{j_{2}}d(a_{j_{1}},a_{j_{2}})^{p} \\
& ~ & - \sum\limits_{1 \leq i_{1} < i_{2} \leq t} n_{i_{1}}n_{i_{2}}d(b_{i_{1}},b_{i_{2}})^{p}
\end{eqnarray*}
for each $\vec{\omega}=(m_{1}, \ldots, m_{s}, n_{1}, \ldots, n_{t}) \in N_{s,t}$.
\end{defn}
\noindent Notice that $\gamma_{D}^{p}(\vec{\omega})$ takes the difference between the right side
and the left side of the inequality (\ref{TWO}). So, by Theorem \ref{REMGR}, $(X,d)$ has
strict $p$-negative type if and only if $\gamma_{D}^{p}(\vec{\omega}) > 0$ for each normalized
$(s,t)$-simplex $D(\vec{\omega})$ in $X$.

\begin{defn}\label{NGAP} Let $p \geq 0$.
Let $(X,d)$ be a metric space with $p$-negative type. We define the
\textit{(normalized)} $p$-\textit{negative type gap} of $(X,d)$
to be the non-negative quantity $$\Gamma_{X}^{p} = \inf\limits_{D(\vec{\omega})} \gamma_{D}^{p}(\vec{\omega})$$
where the infimum is taken over all normalized $(s,t)$-simplices $D(\vec{\omega})$ in $X$.
\end{defn}

\noindent Recall that a \textit{finite metric tree} is a finite connected graph that has no cycles,
endowed with an edge weighted path metric.
Hjorth et al.\ \cite{HLM} have shown that
finite metric trees have strict $1$-negative type. Therefore it makes sense to try to
compute the $1$-negative type gap of any given finite metric tree. Indeed, a
very succinct formula was derived in \cite[Corollary 4.13]{DW}.
However, a modicum of additional notation is necessary before
stating this result. The set of all edges in a metric tree $(T,d)$,
considered as unordered pairs, will be denoted $E(T)$, and the metric length $d(x,y)$ of any given
edge $e=(x,y) \in E(T)$ will be denoted $|e|$.

\begin{thm}[Doust and Weston \cite{DW}]\label{treegap}
Let $(T,d)$ be a finite metric tree. Then
the (normalized) $1$-negative type gap $\Gamma = \Gamma_{T}^{1}$ of $(T,d)$ is given by the following formula:
\[
\Gamma = \Biggl\{ \sum\limits_{e \in E(T)} |e|^{-1} \Biggl\}^{-1}.
\]
In particular, $\Gamma > 0$.
\end{thm}
\noindent In the remaining sections of this paper we shall show how the notions, equivalences and
results of this section may be used to infer some basic properties of metrics of strict $p$-negative
metrics for general values of $p \geq 0$.

\section{A quantitative lower bound on supremal strict $p$-negative type}
The observation is made in \cite[Theorem 5.2]{DW} that if the $p$-negative type gap
$\Gamma_{X}^{p}$ of a finite metric space $(X,d)$ is positive for some $p \geq 0$, then $(X,d)$
must have strict $q$-negative type on some interval of the form $[p, p + \zeta)$ where $\zeta > 0$.
However, the authors only provide an explicit value for $\zeta$ in the case $p = 1$. Letting $n=|X|$,
the value of $\zeta$ given in this case is $O(1/n^{2})$. (See \cite[Theorem 5.1]{DW}.)
The purpose of the present section is to give a precise quantitative version of \cite[Theorem 5.2]{DW}
which yields significantly improved values of $\zeta$ for all $p \geq 0$.
In fact, for each $p \geq 0$, our value of $\zeta$ is $O(1)$.
The precise statement of this result is given in Theorem \ref{mainthm}.
As an application we obtain
significantly improved lower bounds on the maximal $p$-negative type
of finite metric trees. These are stated in Corollary \ref{treebound}. Then in Remark \ref{treerem}
we point out that the estimates given in Corollary \ref{treebound} are actually close to best possible
for finite metric trees that resemble stars. This suggests there is little room for improvement
in the statement of Theorem \ref{mainthm}, the main result of this section.

The proof of Theorem \ref{mainthm} is facilitated by the following two technical lemmas which
are easily realized using basic calculus or by simple combinatorial arguments. The proofs of
these lemmas are therefore omitted.

\begin{lem}\label{techlem}
Let $s \in \mathbb{N}$. If $s$ real variables $\ell_{1}, \ldots, \ell_{s} > 0$ are subject to the
constraint $\ell_{1} + \cdots + \ell_{s} = 1$, then the expression
$$ \sum\limits_{k_{1} < k_{2}} \ell_{k_{1}} \ell_{k_{2}} $$
has maximum value $\frac{s(s-1)}{2} \cdot \frac{1}{s^{2}} = \frac{1}{2}(1 - \frac{1}{s})$ which is
attained when $\ell_{1} = \cdots = \ell_{s} = \frac{1}{s}$.
\end{lem}

\begin{lem}\label{zerogap}
Let $s,t \in \mathbb{N}$ and let $m=s+t$. Then
$$ \frac{1}{2}\left(1 - \frac{1}{s}\right) + \frac{1}{2}\left(1 - \frac{1}{t}\right) \leq
1 - \frac{1}{2} \left( \frac{1}{\lfloor \frac{m}{2} \rfloor} +
\frac{1}{\lceil \frac{m}{2} \rceil} \right).$$
Moreover, the function $\gamma(m) = 1 -
\frac{1}{2} \left( \frac{1}{\lfloor \frac{m}{2} \rfloor} +
\frac{1}{\lceil \frac{m}{2} \rceil} \right)$ increases strictly as $m$ increases.
\end{lem}

\noindent We will continue to use the notation $\gamma(m) = 1 - \frac{1}{2} \cdot
\bigl( {\lfloor \frac{m}{2} \rfloor}^{-1} + {\lceil \frac{m}{2} \rceil}^{-1} \bigl)$
introduced in the preceding lemma throughout the remainder of this section as it
allows the efficient statement and succinct proof of certain key formulas such as Theorem \ref{mainthm}.

The following basic notions are also relevant to the proof of Theorem \ref{mainthm}.
Let $(X,d)$ be a metric space.
If $d(x,y) = 1$ for all $x \not= y$, then $d$ is called the \textit{discrete metric} on $X$.
The \textit{metric diameter} of $(X,d)$ is given by the quantity
$\text{diam}\, X = \sup \{ d(x,y) | x,y \in X \}$. Provided $|X| < \infty$, 
the \textit{scaled metric diameter} of $(X,d)$ is given by the ratio
$\mathfrak{D}_{X} = (\text{diam}\, X) / \min \{d(x,y) | x \not= y\}$.

\begin{thm}\label{mainthm}
Let $(X,d)$ be a finite metric space with cardinality $n = |X| \geq 3$ and let $p \geq 0$.
If the $p$-negative type gap $\Gamma_{X}^{p}$ of $(X,d)$ is positive, then
$(X,d)$ has $q$-negative type for all $q \in [p, p + \zeta]$ where
\begin{eqnarray*}
\zeta & = & \frac{\ln \biggl( 1 + \frac{\Gamma_{X}^{p}}{ ({\rm{diam}}\, X)^{p}
\cdot  \gamma(n) } \biggl)}{\ln \mathfrak{D}_{X}}.
\end{eqnarray*}
Moreover, $(X,d)$ has strict $q$-negative type for all $q \in [p, p + \zeta)$. In particular,
$p + \zeta$ provides a lower bound on the supremal (strict) $q$-negative type of $(X,d)$.
\end{thm}

\begin{proof} For notational ease we set $\Gamma = \Gamma_{X}^{p}$
and $\mathfrak{D} = \mathfrak{D}_{X}$ throughout this proof.
We may assume that the metric $d$ is not a positive multiple of the discrete metric on $X$. Otherwise, $(X,d)$ would
have strict $q$-negative type for all $q \geq 0$. Hence $\mathfrak{D} > 1$.

Since scaling the metric by a positive constant has no effect on whether the space has $p$-negative type,
we may assume that $\min \{d(x,y) | x \not= y\} = 1$.
This means that $\mathfrak{D}$ is now the diameter of our rescaled metric space (which we will continue
to denote by $(X,d)$). Moreover, for all $\ell = d(x,y) \not= 0$ and all $\zeta > 0$, we have
$\ell^{p+\zeta} - \ell^{p} \leq \mathfrak{D}^{p +\zeta} - \mathfrak{D}^{p}$.
This is because, for any fixed $\zeta > 0$, the function $f(x) = x^{p +\zeta} - x^{p}$
is increasing on the interval $[1, \infty)$. This inequality will be used in the derivation of (\ref{TWELVE}) below.

Consider an arbitrary normalized $(s,t)$-simplex $D = [a_{j}(m_{j});b_{i}(n_{i})]_{s,t}$ in $X$. Necessarily,
$m = s+t \leq n$. For any given $r \geq 0$, let
\begin{eqnarray*}
L (r) & = & \sum\limits_{j_{1}<j_{2}} m_{j_{1}}m_{j_{2}}d(a_{j_{i}},a_{j_{2}})^{r}
+ \sum\limits_{i_{1}<i_{2}} n_{i_{1}}n_{i_{2}}d(b_{i_{1}},b_{i_{2}})^{r},\,\,{\rm{and}} \\
R (r) & = & \sum\limits_{j,i} m_{j}n_{i}d(a_{j},b_{i})^{r}.
\end{eqnarray*}
By definition of the $p$-negative type gap $\Gamma$ we have
\begin{eqnarray}\label{TEN}
L(p) + \Gamma \leq R(p).
\end{eqnarray}
The strategy of the proof is to argue that
\begin{eqnarray}\label{ELEVEN}
L(p + \zeta) < L(p) + \Gamma
& \text{ and } &
R(p) \leq R(p + \zeta)
\end{eqnarray}
provided $\zeta > 0$ is sufficiently small. If so, then $L(p + \zeta) < R(p + \zeta)$
by (\ref{TEN}) and (\ref{ELEVEN}). In other words, $(X,d)$ has strict $(p+\zeta)$-negative type
under these circumstances.
Now, as all non-zero distances in $(X,d)$ are at least one, we automatically obtain the second inequality
of (\ref{ELEVEN}) for all $\zeta > 0$.
Therefore we only need to concentrate on the first inequality of (\ref{ELEVEN}). First of all, notice that
\begin{eqnarray}\label{TWELVE}
L (p + \zeta) - L (p) & = & \sum\limits_{j_{1}<j_{2}} m_{j_{1}}m_{j_{2}}
\bigl( d(a_{j_{1}},a_{j_{2}})^{p + \zeta} - d(a_{j_{1}},a_{j_{2}})^{p} \bigl) \nonumber \\
& ~ & ~ \nonumber \\
& ~ & + \sum\limits_{i_{1}<i_{2}} n_{i_{1}}n_{i_{2}}
\bigl( d(b_{i_{1}},b_{i_{2}})^{p + \zeta} - d(b_{i_{1}},b_{i_{2}})^{p} \bigl) \nonumber \\
& \leq & \Biggl(\sum\limits_{j_{1}<j_{2}} m_{j_{1}}m_{j_{2}}
+ \sum\limits_{i_{1}<i_{2}} n_{i_{1}}n_{i_{2}} \Biggl)
\cdot \bigl( \mathfrak{D}^{p + \zeta} - \mathfrak{D}^{p} \bigl) \nonumber \\
& \leq & \Biggl( 1 - \frac{1}{2} \left( \frac{1}{s} + \frac{1}{t} \right) \Biggl)
\cdot \bigl( \mathfrak{D}^{p + \zeta} - \mathfrak{D}^{p} \bigl) \nonumber \\
& \leq & \Biggl( 1 - \frac{1}{2} \left( \frac{1}{\lfloor \frac{m}{2} \rfloor}
+ \frac{1}{\lceil \frac{m}{2} \rceil} \right) \Biggl)
\cdot \bigl( \mathfrak{D}^{p + \zeta} - \mathfrak{D}^{p} \bigl) \nonumber \\
& = & \gamma(m) \cdot \bigl( \mathfrak{D}^{p + \zeta} - \mathfrak{D}^{p} \bigl) \nonumber \\
& \leq & \gamma(n)
\cdot \bigl( \mathfrak{D}^{p + \zeta} - \mathfrak{D}^{p} \bigl),
\end{eqnarray}
by applying Lemmas \ref{techlem} and \ref{zerogap}.
Now observe that:
\begin{eqnarray}\label{THIRTEEN}
\gamma(n) \cdot \bigl( \mathfrak{D}^{p + \zeta} - \mathfrak{D}^{p} \bigl) \leq \Gamma
& \text{iff} & \zeta \leq \frac{\ln \biggl( 1 + \frac{\Gamma}{ \mathfrak{D}^{p} \cdot
\gamma(n)} \biggl)}{\ln \mathfrak{D}}.
\end{eqnarray}
By combining (\ref{TWELVE}) and (\ref{THIRTEEN}), we obtain the first inequality of (\ref{ELEVEN})
for all $\zeta > 0$ such that $$\zeta < \zeta_{0} =
\frac{\ln \biggl( 1 + \frac{\Gamma}{ \mathfrak{D}^{p} \cdot \gamma(n)} \biggl)}{\ln \mathfrak{D}}.$$
Hence $L(p + \zeta) < R(p + \zeta)$ for any such $\zeta$.
It is also clear from (\ref{ELEVEN}), (\ref{TWELVE}) and (\ref{THIRTEEN}) that $L(\zeta_{0}) \leq R(\zeta_{0})$.
These observations and descaling the metric (if necessary) complete the proof of the theorem.
\end{proof}

\noindent Recall that the \textit{ordinary path metric} on a finite tree $T$ assigns length one to
each edge in the tree (with all other distances determined geodesically). With this in mind, we see
that Theorem \ref{mainthm} provides a significant improvement of the estimate given in \cite[Corollary 5.5]{DW}.

\begin{cor}\label{treebound}
Let $T$ be a finite tree on $n = |T| \geq 3$ vertices that is endowed with the ordinary path metric $d$.
Let $\mathfrak{D}$ denote the metric diameter of the resulting finite metric tree $(T,d)$.
Let $\wp_{T}$ denote the maximal $p$-negative type of $(T,d)$. Then:
\begin{eqnarray}\label{bound}
\wp_{T} & \geq & 1 + 
\biggl\{ {\ln \biggl( 1 + \frac{1}{\mathfrak{D} \cdot (n-1) \cdot \gamma(n)} \biggl)}\Bigl/{\ln \mathfrak{D}} \biggl\}.
\end{eqnarray}
\end{cor}

\begin{proof}
By Theorem \ref{treegap}, $\Gamma_{T}^{1} = \frac{1}{n-1}$. Now apply Theorem \ref{mainthm} with $p=1$.
\end{proof}

\begin{rem}\label{treerem}
The lower bound on $\wp_{T}$ given in the statement of Corollary \ref{treebound} is basically of
the correct order of magnitude when $\mathfrak{D} = 2$. To see this, first of all notice that
if $n > 2$ is even and $\mathfrak{D} = 2$, then (\ref{bound}) in Corollary \ref{treebound} simplifies to give:
$$\wp_{T} \geq 1 + \biggl\{ \ln \biggl( 1 + \frac{n}{2(n-1)(n-2)} \biggl)\Bigl/{\ln 2} \biggl\}.$$
However, if $T$ denotes a star with $n-1$ leaves (endowed with the ordinary path metric),
then \cite[Theorem 5.6]{DW} gives the exact value:
$$\wp_{T} = 1 + \biggl\{ {\ln \biggl( 1 + \frac{1}{n-2} \biggl)}\Bigl/{\ln 2} \biggl\}.$$
\end{rem}

\section{Supremal $p$-negative type of a finite metric space cannot be strict}
If the $p$-negative type gap $\Gamma_{X}^{p}$ of a metric space $(X,d)$ is positive then
$(X,d)$ clearly has strict $p$-negative type.
It is interesting to ask to what extent --- if any --- the converse of this statement is true.
Our next result points out that the converse statement is
always true in the case of finite metric spaces. By way of a notable contrast,
\cite[Theorem 5.7]{DW} shows that there exist infinite metric trees $(X,d)$ of strict
$1$-negative type with $1$-negative type gap $\Gamma_{X}^{1}=0$.

\begin{thm}\label{mainthm2}
Let $p \geq 0$ and let $(X,d)$ be a finite metric space. Then $(X,d)$ has strict $p$-negative type
if and only if $\Gamma_{X}^{p} > 0$.
\end{thm}

\begin{proof} Let $p \geq 0$ be given.
We need only concern ourselves with the forward implication of the theorem since the converse is clear
from the definitions.

Assume that $(X,d)$ is a finite metric space with strict $p$-negative type. By Theorem \ref{REMGR},
$\gamma_{D}^{p}(\vec{\omega}) > 0$ for each normalized $(s,t)$-simplex $D(\vec{\omega}) \subseteq X$.
Referring back to Definitions \ref{CUTIE} and \ref{SGAP} we further note that we may assume that each
such $p$-negative type simplex gap $\gamma_{D}^{p}$ is defined on the compact set
$\overline{N_{s,t}} \subset \mathbb{R}^{s+t}$ and is positive at each point of $\overline{N_{s,t}}$.
Therefore
$$\min \biggl\{ \, \gamma_{D}^{p}(\vec{\omega}) \, | \, \vec{\omega} \in \overline{N_{s,t}} \, \biggl\} > 0$$
for each $(s,t)$-simplex $D$ in $X$.
But as $|X| < \infty$ the number of distinct $(s,t)$-simplexes $D$
that can be formed from $X$ must be finite. Thus the $p$-negative type gap $\Gamma_{X}^{p}$
is seen to be the minimum of finitely many positive quantities. As such we obtain
the desired result: $\Gamma_{X}^{p} > 0$.
\end{proof}

\begin{cor}\label{maincor}
Let $p \geq 0$ and let $(X,d)$ be a finite metric space. If $(X,d)$ has strict $p$-negative type,
then $(X,d)$ must have strict $q$-negative type for some interval of values $q \in [p,p+\zeta)$, $\zeta > 0$.
\end{cor}

\begin{proof}
By Theorem \ref{mainthm2}, $\Gamma = \Gamma_{X}^{p} > 0$. Now apply Theorem \ref{mainthm}.
\end{proof}

\noindent As an immediate consequence of Corollary \ref{maincor} we obtain one of the main results
of this paper.

\begin{cor}\label{strict}
The supremal $p$-negative type of a finite metric space cannot be strict.
\end{cor}

\noindent Moreover, since $p$-negative type holds on closed intervals, we therefore obtain
an interesting case of equality in the classical negative type inequalities as
a direct consequence of Corollary \ref{strict}.

\begin{cor}\label{equality}
Let $(X,d)$ be a finite metric space. Let $\wp$ denote the supremal $p$-negative type of $(X,d)$.
If $\wp < \infty$ then there exists a normalized $(s,t)$-simplex $D(\vec{\omega}) = [a_{j}(m_{j});b_{i}(n_{i})]_{s,t}$
in $X$ such that $\gamma_{D}^{\wp}(\vec{\omega}) = 0$. In other words, we obtain:
\begin{eqnarray*}
& & \sum\limits_{1 \leq j_{1} < j_{2} \leq s} m_{j_{1}}m_{j_{2}}d(a_{j_{1}},a_{j_{2}})^{\wp} +
\sum\limits_{1 \leq i_{1} < i_{2} \leq t} n_{i_{1}}n_{i_{2}}d(b_{i_{1}},b_{i_{2}})^{\wp} \\
& = & \sum\limits_{j,i=1}^{s,t} m_{j}n_{i}d(a_{j},b_{i})^{\wp}.
\end{eqnarray*}
\end{cor}

\begin{cor}\label{strictexs}
The following finite metric spaces all have strict $q$-negative type for some interval of values
$q \in [1, 1 + \zeta)$ (where $\zeta > 0$ depends upon the particular space):
\begin{enumerate}
\item[(a)] Any three-point metric space.

\item[(b)] Any finite metric tree.

\item[(c)] Any finite isometric subspace of a $k$-sphere $\mathbb{S}^{k}$ (endowed with the usual
geodesic metric) that contains at most one pair of antipodal points.

\item[(d)] Any finite isometric subspace of the hyperbolic space $\mathbb{H}_{\mathbb{R}}^{k}$
(or $\mathbb{H}_{\mathbb{C}}^{k}$).

\item[(e)] Any finite isometric subspace of a Hadamard manifold.
\end{enumerate}
\end{cor}

\begin{proof}
All of the above finite metric spaces have strict $p$-negative type for $p=1$ by results given
in \cite{HKM} and \cite{HLM}. We may therefore apply Corollary \ref{maincor} \textit{en masse}.
\end{proof}

\section{Range of strict $p$-negative type}
It is a classical result of Schoenberg \cite[Theorem 2]{S2} that $p$-negative type holds
on closed intervals. More precisely, the set of all values of $p$ for which a given metric space $(X,d)$
has $p$-negative type is always an interval of the form $[0,\wp]$ or $[0,\infty)$.
Included here is the possibility that $\wp = 0$, in which case the interval degenerates to $\{ 0 \}$.
Examples of Enflo \cite{E} in tandem with Theorem \ref{REMGR} (a) imply that all
such intervals (degenerate or otherwise) can occur. Moreover, for intervals of the form
$[0, \wp]$ with $\wp > 0$, the examples given in \cite[Section 1]{E} are finite metric spaces.
In the case of the degenerate interval $\{ 0 \}$ the situation is slightly more delicate.
It follows from \cite[Theorem 2.1]{E} and Theorem \ref{REMGR} (a) that the Banach space $C[0,1]$
does not have $p$-negative type for any $p >0$.
In Theorems \ref{sint}, \ref{finint} and \ref{infint} we provide strict versions of \cite[Theorem 2]{S2}.
These theorems allow us to precisely codify the types of intervals on which a metric space
may have strict $p$-negative type. It is interesting to note that finite metric spaces
behave quite differently to infinite metric spaces in this respect. Theorems \ref{finint} and \ref{infint}
highlight this point.

In order to proceed we must first briefly recall some basic facts about kernels of positive type and kernels
conditionally of negative type. (In some important respects we are following Nowak \cite[Sections 2--4]{N}.)

\begin{defn} Let $X$ be a topological space.
\begin{enumerate}
\item[(a)] A kernel of \textit{positive type} on $X$ is a continuous function
$\Phi: X \times X \rightarrow \mathbb{C}$ such that for any $n \in \mathbb{N}$, any elements
$x_{1}, \ldots, x_{n} \in X$, and any complex numbers $\eta_{1}, \ldots, \eta_{n}$ we have:
\[
\sum\limits_{1 \leq i,j \leq n} \Phi(x_{i},x_{j}) \eta_{i} \overline{\eta_{j}} \geq 0.
\]

\item[(b)] A kernel \textit{conditionally of negative type} on $X$ is a continuous function
$\Psi: X \times X \rightarrow \mathbb{R}$ with the following three properties:
\begin{enumerate}
\item[(1)] $\Psi(x,x) = 0$ for all $x \in X$,

\item[(2)] $\Psi(x,y) = \Psi(y,x)$ for all $x,y \in X$, and

\item[(3)] for any $n \in \mathbb{N}$, any $x_{1}, \ldots, x_{n} \in X$, and any real numbers
$\eta_{1}, \ldots, \eta_{n}$ with $\eta_{1} + \cdots + \eta_{n} = 0$ we have:
\[
\sum\limits_{1 \leq i,j \leq n} \Psi(x_{i},x_{j}) \eta_{i}\eta_{j} \leq 0.
\]
\end{enumerate}
\end{enumerate}
\end{defn}

\noindent The following fundamental relationship between kernels of positive type and kernels conditionally of
negative type was given by Schoenberg \cite{S3}. For a short proof of this theorem we refer the reader
to Bekka et al.\ \cite[Theorem C.3.2]{BHV}.

\begin{thm}[Schoenberg \cite{S3}]\label{Scho}
Let $X$ be a topological space and $\Psi: X \times X \rightarrow \mathbb{R}$ be a continuous
kernel on $X$ such that $\Psi(x,x) = 0$ and $\Psi(x,y) = \Psi(y,x)$ for all $x,y \in X$. Then
$\Psi$ is conditionally of negative type if and only if the kernel $\Phi = e^{-t \Psi}$ is of positive
type for every $t \geq 0$.
\end{thm}

\noindent Our proof of Theorem \ref{sint} makes use of the following
identity. An explanation of this identity may be found in the
proof of Corollary 3.2.10 in Berg et al.\ \cite{BCR}.

\begin{lem}\label{intlem}
For each $\alpha \in (0,1)$ there exists a constant $c_{\alpha}>0$ such that 
$$x^{\alpha} = c_{\alpha} \int\limits_{0}^{\infty} (1 - e^{-tx}) t^{- \alpha - 1} dt$$
for all $x\ge 0$.
\end{lem}

\begin{thm}\label{sint}
Let $(X,d)$ be a metric space. If $(X,d)$ has $p$-negative type for some $p > 0$, then it must have
strict $q$-negative type for all $q$ such that $0 \leq q < p$.
\end{thm}

\begin{proof}
Every metric space has strict $0$-negative type.
So we may assume that $q > 0$.
Since $(X, d)$ has $p$-negative type, the function $\Psi: X\times X\rightarrow \Rb$ defined by $\Psi(x, y)=d(x, y)^p$
is conditionally of negative type. Hence, by Theorem \ref{Scho}, the function
$e^{-t\Psi}: X\times X\rightarrow \Cb$ is of positive type for every $t\ge 0$.

Let $x_1, \ldots , x_n$ ($n \geq 2$) be distinct
points in $X$ and let $\eta_1, \ldots, \eta_n$ be real numbers, not all zero, such that $\sum_j\eta_j=0$.
We need to show that $\sum_{i, j}d(x_i, x_j)^q\eta_i\eta_j < 0$.

For each $t \ge 0$, set $$f(t)=\sum_{i,j}(1-e^{-td(x_i, x_j)^p})\eta_i\eta_j.$$ 
Then
\begin{eqnarray*}
f(t)&=&\sum_{i, j}\eta_i\eta_j-\sum_{i, j}e^{-td(x_i, x_j)^p}\eta_i\eta_j
=\biggl(\sum_j\eta_j\biggl)^2-\sum_{i, j}e^{-td(x_i, x_j)^p}\eta_i\eta_j\\
&=&
-\sum_{i, j}e^{-td(x_i, x_j)^p}\eta_i\eta_j\le 0
\end{eqnarray*}
for all $t\ge 0$. When $t\to \infty$, one has $f(t)\to -\sum_j\eta^2_j<0$. Thus $f(t)<0$ for all $t$ sufficiently large. 
Set $\alpha=q/p$. By applying Lemma \ref{intlem} to $x=d(x_i, x_j)^p$, one gets 
\begin{eqnarray*}
\sum_{i, j}d(x_i,x_j)^q\eta_i\eta_j&=&\sum_{i, j}
\biggl(c_{\alpha}\int^{\infty}_0(1-e^{-td(x_i, x_j)^p})t^{-\alpha-1}\, dt \biggl)\eta_i\eta_j\\
&=&
c_{\alpha}\int^{\infty}_0f(t)t^{-\alpha-1}dt<0,
\end{eqnarray*}
as desired.
\end{proof}

\noindent As an immediate consequence of Corollary \ref{strict}, Theorem \ref{sint}, \cite[Theorem 4.3]{W}
and the examples given in \cite[Section 1]{E} we obtain the following theorem.

\begin{thm}\label{finint}
Let $(X,d)$ be a finite metric space.
The set of all values of $p$ for which $(X,d)$ has strict $p$-negative type
is always an interval of the form $[0,\wp)$, with $\wp > 0$, or $[0,\infty)$.
Moreover, all such intervals can occur.
\end{thm}

\noindent By way of marked contrast with Theorem \ref{finint}, we note that (a) the set of
all values of $p$ for which the Banach space $C[0,1]$ has strict $p$-negative type is the
degenerate interval $\{ 0 \}$ (this follows from \cite[Theorem 2.1]{E} and Theorem
\ref{REMGR}), and (b) the supremal $p$-negative type of an infinite metric space may or
may not be strict. For example, in \cite[Theorem 5.7]{DW} the authors construct an infinite
metric tree that has strict $p$-negative type if and only if $p \in [0,1]$. However, the
Banach space $\ell_{1}$ has strict $p$-negative type if and only if $p \in [0,1)$.
Our next theorem points out that for each $\wp > 0$ there is an
infinite metric space $(X,d)$ that has strict $p$-negative type if and only if $p \in [0,\wp]$.
The following definition introduces the relevant spaces.

\begin{defn}\label{penflo}
Let $\wp > 0$. Let $(\wp_{k})$ be a strictly decreasing sequence of real numbers that converges to $\wp$.
Let $n$ be a natural number that satisfies the condition: $$\left(1 - \frac{1}{n}\right)^{1/\wp} \geq 1/2.$$
Let $Y$ be the union of a sequence of pairwise disjoint sets $(Y_{1},Y_{2},Y_{3}, \ldots)$
such that $|Y_{k}|=n$ for all $k \in \mathbb{N}$.
Let $Z$ be the union of a sequence of pairwise disjoint sets $(Z_{1},Z_{2},Z_{3}, \ldots)$
such that $|Z_{k}|=n$ for all $k \in \mathbb{N}$ and $Y \cap Z = \emptyset$.
Set $X = Y \cup Z$. We metrize $X$ in the following way:
\[
d(y,z) = \left( 1 - \frac{1}{n} \right)^{1 / \wp_{k}}
\]
if $y \in Y_{k}$ and $z \in Z_{k}$ for some $k \in \mathbb{N}$.
All other non-zero distances in the space are taken to be one.
We call $(X,d)$ an \textit{Enflo $\wp$-space}.
\end{defn}

\noindent Notice that in Definition \ref{penflo} the condition placed on $n$ ensures that
$(X,d)$ really is a metric space. Moreover, from the examples noted in \cite[Section 1]{E},
together with Theorem \ref{REMGR} (a), it follows that each subspace $Y_{k} \cup Z_{k}$ of
an Enflo $\wp$-space $(X,d)$ has maximal $p$-negative type exactly equal to $\wp_{k}$.
In order to proceed we need to develop a slightly stronger statement about certain subspaces of $(X,d)$.

\begin{lem}\label{sub}
Let $\wp > 0$. Let $(X,d)$ be an Enflo $\wp$-space as in Definition \ref{penflo}.
For each $m \in \mathbb{N}$, the subspace $X_{m}
= \bigcup \{ Y_{k} \cup Z_{k} | 1 \leq k \leq m \}$ of $(X,d)$ has $\wp_{m}$-negative type.
In fact, $\wp_{m}$ is the maximal $p$-negative type of the subspace $X_{m}$.
\end{lem}

\begin{proof}
Let $D$ be a given normalized simplex in $X_{m}$. Without loss of generality, we may assume that
$D=[y_{k,j}(\alpha_{k,j}), z_{k,j}(\beta_{k,j}); \bar{y}_{k,j}(\gamma_{k,j}),
\bar{z}_{k,j}(\eta_{k,j})]$, where $y_{k,j}, \bar{y}_{k,j} \in Y_{k}$ and
$z_{k,j}, \bar{z}_{k,j} \in Z_{k}$. In other words, our
simplex has the points $y_{k,j}, z_{k,j}$ on one side and the remaining points
$\bar{y}_{k,j}, \bar{z}_{k,j}$ on the other side (with weights as indicated).

Let $L(\wp_{m})$ and $R(\wp_{m})$ denote the left side and right side of (\ref{TWO})
computed relative to the simplex $D$ with exponent $\wp_{m}$ (respectively). By setting
$\alpha_{k} = \sum_{j} \alpha_{k,j}, \beta_{k} = \sum_{j} \beta_{k,j}, \gamma_{k}
= \sum_{j} \gamma_{k,j}$ and $\eta_{k} = \sum_{j} \eta_{k,j}$ we may then
compute $L(\wp_{m})$ and $R(\wp_{m})$:

\begin{eqnarray}\label{EIGHT}
L(\wp_m) & = & \sum_k(\alpha_k \beta_k+\gamma_k
\eta_k)\biggl\{\biggl(1-\frac{1}{n}\biggl)^{\wp_m/\wp_k}-1\biggl\} + 1 \nonumber \\
& & -\frac{1}{2}\sum_k\biggl(\sum_j\alpha_{k,j}^2+\sum_j\beta_{k, j}^2
+\sum_j\gamma_{k, j}^2+\sum_j\eta_{k, j}^2\biggl),
\end{eqnarray}
while
\begin{eqnarray}\label{NINE}
R(\wp_m) & = & \sum_k(\alpha_k \eta_k+\gamma_k
\beta_k)\biggl\{\biggl(1-\frac{1}{n}\biggl)^{\wp_m/\wp_k}-1\biggl\}+1.
\end{eqnarray}

If we let $s_{k}$ denote the number of points $y_{k,j}$ in $Y_{k}$ and
let $t_{k}$ denote the number of points $\bar{y}_{k,j}$ in $Y_{k}$,
then it follows that we have:
\begin{eqnarray*}
\sum\limits_{j} \alpha_{k,j}^{2} +
\sum\limits_{j} \gamma_{k,j}^{2} & \geq & \frac{\alpha_k^2}{s_k}+\frac{\gamma_{k}^2}{t_k} \\
                                 & \geq & \frac{\alpha_k^2+\gamma_k^2}{n}.
\end{eqnarray*}
Similarly,
\begin{eqnarray*}
\sum\limits_{j} \beta_{k,j}^{2} + \sum\limits_{j} \eta_{k,j}^{2}
& \geq & \frac{\beta_{k}^{2} + \eta_{k}^{2}}{n}.
\end{eqnarray*}

As a result, comparing the expressions (\ref{EIGHT}) and (\ref{NINE}), we see that it
will follow that $L(\wp_{m}) \leq R(\wp_{m})$ provided that we can establish the following
non-linear inequality:

\begin{eqnarray}\label{ELEVENa}
& & \sum\limits_{k} \biggl\{ (\alpha_{k}\beta_{k} + \gamma_{k}\eta_{k})
- (\alpha_{k} \eta_{k} + \gamma_{k} \beta_{k}) \biggl\}
\biggl\{ \biggl( 1 - \frac{1}{n} \biggl)^{\wp_{m}/\wp_{k}} -1 \biggl\} \nonumber \\
& \leq & \frac{1}{2}\sum_k\frac{\alpha_k^2+\gamma_k^2+\beta_k^2+\eta_k^2}{n}.
\end{eqnarray}

We claim that (\ref{ELEVENa}) holds term by term. That is to say,
\begin{eqnarray}\label{TWELVEa}
& & \biggl\{ (\alpha_{k}\beta_{k} + \gamma_{k}\eta_{k})
- (\alpha_{k} \eta_{k} + \gamma_{k} \beta_{k}) \biggl\}
\biggl\{ \biggl( 1 - \frac{1}{n} \biggl)^{\wp_{m}/\wp_{k}} -1 \biggl\} \nonumber \\
& \leq & \frac{1}{2} \cdot \frac{\alpha_k^2+\gamma_k^2+\beta_k^2+\eta_k^2}{n},
\end{eqnarray}
for each $k$. In fact, for each $k$,
\begin{eqnarray*}
& & \biggl| \biggl\{ (\alpha_{k}\beta_{k} + \gamma_{k}\eta_{k})
- (\alpha_{k} \eta_{k} + \gamma_{k} \beta_{k}) \biggl\}
\biggl\{ \biggl( 1 - \frac{1}{n} \biggl)^{\wp_{m}/\wp_{k}} -1 \biggl\} \biggl| \\ 
& \leq & \max \bigl( \alpha_{k} \beta_{k} + \gamma_{k} \eta_{k},
                 \alpha_{k} \eta_{k} + \gamma_{k} \beta_{k} \bigl) \cdot \frac{1}{n} \\ 
& \leq & \frac{1}{2} \cdot \frac{\alpha_{k}^{2} + \beta_{k}^{2} + \gamma_{k}^{2} + \eta_{k}^{2}}{n}.
\end{eqnarray*}

\noindent As (\ref{TWELVEa}) implies (\ref{ELEVENa}) we conclude that
$L(\wp_{m}) \leq R(\wp_{m})$. Thus $X_{m}$ has maximal $p$-negative type at least $\wp_{m}$.
However the subspace $Y_{m} \cup Z_{m}$ of $X_{m}$ has maximal $p$-negative type $\wp_{m}$
by \cite[Section 1]{E} and Theorem \ref{REMGR} (a). We conclude that the maximal $p$-negative
type of $X_{m}$ is $\wp_{m}$.
\end{proof}

\begin{thm}\label{infint}
Let $\wp > 0$. Let $(X,d)$ be an Enflo $\wp$-space as in Definition \ref{penflo}.
Then $(X,d)$ has strict $p$-negative type if and only if
$p \in [0, \wp]$.
\end{thm}

\begin{proof}
For each $k$ the subspace $Y_{k} \cup Z_{k}$ of $(X,d)$ has maximal $p$-negative type $\wp_{k}$.
And since $\wp_{k} \searrow \wp$ as $k \rightarrow \infty$ it follows that $(X,d)$ does
not have $p$-negative type for any $p > \wp$.

However, each subspace $X_{m} = Y_{m} \cup Z_{m}$ of $(X,d)$ has $\wp_{m}$-negative type by Lemma \ref{sub}.
By Theorem \ref{sint}, $X_{m}$ has strict $p$-negative type for all $p \in [0,\wp_{m})$. Thus
$(X,d)$ has strict $p$-negative type for all $p \in [0,\wp]$ as asserted.
\end{proof}

\noindent We conclude this paper with some final applications of Theorem \ref{sint}. Recall that
the maximal $q$-negative type of certain classical (quasi-) Banach spaces has been computed explicitly.
For example, suppose $0 < p \leq 2$ and that $\mu$ is a non-trivial positive measure, then the maximal $q$-negative type
of $L_{p}(\mu)$ is simply $p$. A short proof of this result, which is due to Schoenberg \cite{S3} for
$1 \leq p \leq 2$, may be found in \cite[Corollary 2.6 (a)]{LTW}. Theorem \ref{sint}
therefore applies as follows.

\begin{cor}\label{Lp}
Let $0 < p \leq 2$ and let $\mu$ be a positive measure. Then any metric space $(X,d)$ which is isometric
to a subset of $L_{p}(\mu)$ must have strict $q$-negative type for all $q \in [0,p)$.
\end{cor}

\noindent Corollary \ref{equality} and Theorem \ref{sint} combine to provide the following characterization
of the supremal $p$-negative type of a finite metric space in terms of zeros of the simplex gap
functions $\gamma_{D}^{q}$.

\begin{cor}\label{finite}
If the supremal $p$-negative type $\wp$ of a finite metric space $(X,d)$ is finite, then:
\[
\wp = \min \{ q \, | \, q > 0 \text{ and } \gamma_{D}^{q}(\vec{\omega}) = 0 \text{ for some normalized (s,t)-simplex }
D(\vec{\omega}) \subseteq X \}.
\]
\end{cor}

\noindent In certain instances Theorem \ref{sint} provides a second description of the maximal $p$-negative type
of a metric space.

\begin{cor}\label{max}
Let $\wp > 0$. If a metric space $(X,d)$ has $\wp$-negative type but not strict $\wp$-negative type, then
$\wp$ is the maximal $p$-negative type of $(X,d)$.
\end{cor}

\noindent It follows from Kelly \cite{K1, K2}
that any $k$-sphere $\mathbb{S}^{k}$ endowed with the usual geodesic metric
is $\ell_{1}$-embeddable and therefore of $1$-negative type.
On the other hand, Hjorth et al.\ \cite[Theorem 9.1]{HLM} have shown that a finite
isometric subspace $(X,d)$ of a $k$-sphere $\mathbb{S}^{k}$ is of strict $1$-negative type
if and only if $X$ contains at most one pair of antipodal points. These comments and Corollary \ref{max}
imply the next corollary.

\begin{cor}
A finite isometric subspace $(X,d)$ of a $k$-sphere $\mathbb{S}^{k}$ has maximal $p$-negative type $= 1$
if and only if $X$ contains at least two pairs of antipodal points.
\end{cor}

\noindent Recall, following Blumenthal \cite{B}, that a \textit{semi-metric space} is required to satisfy
all of the axioms of a metric space except (possibly) the triangle inequality.

\begin{rem}\label{lastrem}
In closing we note that Theorems \ref{mainthm}, \ref{mainthm2} and \ref{sint} hold (more generally) for all
finite semi-metric spaces $(X,d)$. The same goes for Corollaries \ref{maincor}, \ref{strict}, \ref{equality},
\ref{finite} and \ref{max}. This is because the triangle inequality has played no r\^{o}le in any
of the definitions or computations of this paper.
\end{rem}

\section*{Acknowledgments}
We would like to thank the referee for making detailed comments on the preliminary version
of this paper and --- \textit{in particular} --- for asking for a more in-depth analysis of
the intervals on which a metric space may have (strict) $p$-negative type.

\bibliographystyle{amsalpha}

\begin{thebibliography}{00}
\bibitem{BHV} B. Bekka, P. de la Harpe, and A. Valette. {\sl Kazhdan's Property (T)},
New Mathematical Monographs (No. 11), Cambridge University Press, Cambridge (2008).

\bibitem{BCR} C. Berg, J. P. R. Christensen and P. Ressel, \textsl{Harmonic Analysis on Semigroups
(Theory of Positive Definite and Related Functions)}, Graduate Texts in Mathematics (Vol. 100),
Springer-Verlag, New York (1984).

\bibitem{BL} Y. Benyamini and J. Lindenstrauss, \textsl{Geometric Nonlinear Functional Analysis (Vol.\ 1)},
American Mathematical Society (Providence), American Mathematical Society Colloquium Publications
\textbf{48} (2000).

\bibitem{B} L. M. Blumenthal, \textit{A new concept in distance geometry with applications to spherical subsets},
Bull. Amer. Math. Soc. \textbf{47} (1941), 435--443.

\bibitem{BDK} J. Bretagnolle, D. Dacunha-Castelle and J. L. Krivine, \textit{Lois stables et espaces $L^{p}$},
Ann. Inst. H. Poincar\'{e} Sect. B (N.S.) \textbf{2} (1966), 231--259.

\bibitem{C} A. Cayley, \textit{On a theorem in the geometry of position},
Cambridge Mathematical Journal \textbf{II} (1841),
267--271. (Also in \textsl{The Collected Mathematical Papers of Arthur Cayley (Vol.\ I)}, Cambridge University Press,
Cambridge (1889), pp.\,1--4.)

\bibitem{DL} M. M. Deza and M. Laurent, \textsl{Geometry of Cuts and Metrics}, Springer-Verlag (Berlin),
Algorithms and Combinatorics \textbf{15} (1997).

\bibitem{DW} I. Doust and A. Weston, \textit{Enhanced negative type for finite metric trees}, J. Funct. Anal.
\textbf{254} (2008), 2336--2364. (See \texttt{arXiv:0705.0411v2} for an extended version of this paper.)

\bibitem{DW2} I. Doust and A. Weston, \textit{Corrigendum to ``Enhanced negative type for finite metric
trees''}, J. Funct. Anal. \textbf{255} (2008), 532--533.

\bibitem{DGLY} A. N. Dranishnikov, G. Gong, V. Lafforgue and G. Yu, \textit{Uniform embeddings into Hilbert
space and a question of Gromov}, Can. Math. Bull. \textbf{45} (2002), 60--70.

\bibitem{E} P. Enflo, \textit{On a problem of Smirnov}, Ark. Mat. \textbf{8} (1969), 107--109.

\bibitem{G} M. Gromov, \textit{Asymptotic invariants of infinite groups},
in \textsl{Geometric Group Theory, Vol. 2} (Editors G. A. Niblo and M. A. Roller), Proc. Symp. Sussex, 1991,
London Math. Soc. Lecture Note Ser. \textbf{182}, Cambridge University Press, Cambridge (1993), pp.\,1--295.

\bibitem{HKM} P. G. Hjorth, S. L. Kokkendorff and S. Markvorsen, \textit{Hyperbolic spaces are of
strictly negative type}, Proc. Amer. Math. Soc. \textbf{130} (2002), 175--181.

\bibitem{HLM} P. Hjorth, P. Lison\v{e}k, S. Markvorsen and C. Thomassen, \textit{Finite metric spaces
of strictly negative type}, Lin. Alg. Appl. \textbf{270} (1998), 255--273.

\bibitem{J} M. Junge, \textit{Embeddings of non-commutative $L_{p}$-spaces into non-commutative $L_{1}$-spaces,
$1<p<2$}, Geom. Funct. Anal. \textbf{10} (2000), 389--406.

\bibitem{K1} J. B. Kelly, \textit{Combinatorial inequalities}, in \textsl{Combinatorial Structures
and their Applications} (Editors R. Guy et al.), Gordon and Breach, New York (1970), pp.\,201--208.

\bibitem{K2} J. B. Kelly, \textit{Hypermetric spaces and metric transforms} in
\textsl{Inequalities III} (Edited by O. Shisha), Academic Press, New York (1972), pp.\,149--158.

\bibitem{KV} S. Khot and N. Vishnoi, \textit{The unique games conjecture, integrability gap for cut problems
and embeddings of negative type metrics into $\ell_{1}$}, 46th Annual Symposium on Foundations of Computer
Science, IEEE Computer Society (2005), 53--62.

\bibitem{LN} J. R. Lee and A. Naor, \textit{$L_{p}$ metrics on the Heisenberg group and the Goemans-Linial
conjecture}, preprint. (See \texttt{www.cs.washington.edu/homes/jrl/papers/gl-heisenberg.pdf} for example.)

\bibitem{LTW} C. J. Lennard, A. M. Tonge and A. Weston, \textit{Generalized roundness and negative type},
Mich. Math. J. \textbf{44} (1997), 37--45.

\bibitem{M1} K. Menger, \textit{Untersuchungen \"{u}ber allgemeine Metrik}, Math. Ann. \textbf{100} (1928),
75--163.

\bibitem{M2} K. Menger, \textit{New foundation of Euclidean geometry}, Amer. J. Math. \textbf{53} (1931),
721--745.

\bibitem{M3} K. Menger, \textsl{G\'{e}om\'{e}trie G\'{e}n\'{e}rale}, M\'{e}mor. Sci.
Math., no. 124, Acad\'{e}mie des Sciences de Paris, Gauthier-Villars (1954).

\bibitem{N} P. Nowak, \textit{Coarse embeddings of metric spaces into Banach spaces}, Proc. Amer. Math.
Soc. \textbf{133} (2005), 2589--2596.

\bibitem{PW} E. Prassidis and A. Weston, \textit{Manifestations of non linear roundness in
analysis, discrete geometry and topology}, in \textsl{Limits of Graphs in
Group Theory and Computer Science} (Editors G. Arzhantseva and A. Valette),
{{Research Proceedings of the \'{E}cole Polytechnique F\'{e}d\'{e}rale de Lausanne}}, CRC Press (2009).

\bibitem{S1} I. J. Schoenberg, \textit{Remarks to Maurice Frechet's article ``Sur la d\'{e}finition axiomatique
d'une classe d'espaces distanci\'{e}s vectoriellement applicable sur l'espace de Hilbert.''}, Ann. Math. \textbf{36}
(1935), 724--732.

\bibitem{S2} I. J. Schoenberg, \textit{On certain metric spaces arising from euclidean spaces by a change of metric
and their imbedding in Hilbert space}, Ann. Math. \textbf{38} (1937), 787--793.

\bibitem{S3} I. J. Schoenberg, \textit{Metric spaces and positive definite functions}, Trans. Amer. Math. Soc.
\textbf{44} (1938), 522--536.

\bibitem{WW} J. H. Wells and L. R. Williams, \textsl{Embeddings and Extensions in Analysis}, Ergebnisse
der Mathematik und ihrer Grenzgebiete \textbf{84} (1975).

\bibitem{W} A. Weston, \textit{On the generalized roundness of finite metric spaces}, J. Math. Anal. Appl.
\textbf{192} (1995), 323--334.
\end{thebibliography}

\end{document}